\documentclass{article}
\usepackage{amsfonts}
\usepackage{amssymb}
\usepackage{amsmath}

\newtheorem{theorem}{Theorem}

\newtheorem{corollary}[theorem]{Corollary}

\newtheorem{definition}[theorem]{Definition}
\newtheorem{example}[theorem]{Example}

\newtheorem{proposition}[theorem]{Proposition}
\newtheorem{remark}[theorem]{Remark}

\newenvironment{proof}[1][Proof]{\noindent\textbf{#1.} }{\ \rule{0.5em}{0.5em}}
\numberwithin{theorem}{section}
\numberwithin{equation}{section}
\input{tcilatex}
\begin{document}

\title{The local description of the Ricci and Bianchi identities for an $h$%
-normal $N$-linear connection on the dual $1$-jet space $J^{1\ast }(\mathcal{%
T},M)$}
\author{Alexandru Oan\u{a} and Mircea Neagu}
\date{}
\maketitle

\begin{abstract}
In this paper we describe the local Ricci and Bianchi identities for an $h$%
-normal $N$-linear connection $D\Gamma (N)$ on the dual $1$-jet space $%
J^{1\ast }(\mathcal{T},M)$. To reach this aim, we firstly give the
expressions of the local distinguished (d-) adapted components of torsion
and curvature tensors produced by $D\Gamma (N)$, and then we analyze their
attached local Ricci identities. The derived deflection d-tensor identities
are also presented. Finally, we expose the local expressions of the Bianchi
identities (in the particular case of an $h$-normal $N$-linear connection of
Cartan type), which geometrically connect the local torsion and curvature
d-tensors of the linear connection $D\Gamma (N)$.
\end{abstract}



\section{Introduction}

According to Olver's opinion \cite{Olver}, we consider that the $1$-jet
spaces and their duals are natural houses for the study of classical and
quantum field theories. For such a reason, the differential geometry of $1$%
-jet spaces was intensively studied, in a contravariant approach, by a lot
of authors: Saunders \cite{Saun}, Asanov \cite{Asan}, Neagu and Udri\c{s}te
(see \cite{Neag1}, \cite{Neag+Udri}, \cite{Neag+Udri+Oana}), and many others.

In the last decades, numerous physicists and geometers were preoccupied by
the development of that so-called the \textit{covariant Hamiltonian geometry
of physical fields}, which is a multi-parameter, or multi-time, extension of
the classical Hamiltonian formulation from Mechanics. In such a perspective,
we point out that the covariant Hamiltonian geometry of physical fields
appears in the literature of specialty in three distinct variants: \textbf{%
(1)} $\blacktriangleright $ the \textit{multisymplectic geometry} $-$
developed by Gotay, Isenberg, Marsden, Montgomery and their co-workers (see 
\cite{Gota+Isen+Mars}, \cite{Gota+Isen+Mars+Mont}) on a finite-dimensional
multisymplectic phase space; \textbf{(2)} $\blacktriangleright $ the \textit{%
polysymplectic geometry} $-$ elaborated by Giachetta, Mangiarotti and
Sardanashvily (see \cite{Giac+Mang+Sard1}, \cite{Giac+Mang+Sard2}), which
emphasizes the relations between the equations of first order Lagrangian
field theory on fiber bundles and the covariant Hamilton equations on a
finite-dimensional polysymplectic phase space; \textbf{(3)} $%
\blacktriangleright $ the \textit{De Donder-Weyl Hamiltonian geometry} $-$
studied by Kanatchikov (see \cite{Kana1}, \cite{Kana2}, \cite{Kana3}) as
opposed to the conventional field-theoretical Hamiltonian formalism, which
requires the space + time decomposition and leads to the picture of a field
as a mechanical system with infinitely degrees of freedom.

From a geometrical point of view, following the ideas initially stated by
Asanov \cite{Asan}, a multi-time Lagrange contravariant geometry on $1$-jet
spaces (in the sense of d-linear connections, d-torsions and d-curvatures)
was recently developed by Neagu and Udri\c{s}te in \cite{Neag1}, \cite%
{Neag+Udri} and \cite{Neag+Udri+Oana}. This $1$-jet geometrical theory is a
natural multi-time extension of the classical Lagrangian geometry on tangent
bundles, initiated and developed by Miron and Anastasiei \cite{Miro+Anas}.

On the other hand, suggested by the field theoretical extension of the basic
structures of classical Analytical Mechanics within the framework of the De
Donder-Weyl covariant Hamiltonian formulation, the geometrical studies of
Miron \cite{Miro}, Atanasiu \cite{Atan+Klep}, \cite{Atan1} and others led to
the development of the Hamilton geometry on cotangent bundles, which is
synthesized in the book \cite{Miro+Hrim+Shim+Saba}. Note that the
Miron-Atanasiu Hamiltonian geometrical ideas from cotangent bundles
represent the point start for the development of the jet covariant
Riemann-Hamilton geometry depending on polymomenta, which is presented in
the Atanasiu-Neagu papers \cite{Atan-Neag0} and \cite{Atan+Neag1}. In this
paper we are going on the jet multi-time Hamiltonian geometrical studies
from \cite{Atan-Neag0} and \cite{Atan+Neag1}.

\section{Components of $N$-linear connections on dual $1$-jet bundle $%
J^{1\ast }(\mathcal{T},M)$}

Let $\mathcal{T}$ and $M$ be a \textit{temporal} (resp. \textit{spatial})
real, smooth manifold of dimension $m$ (resp. $n$), whose coordinates are $%
(t^{a})_{a=\overline{1,m}}$, respectively $(x^{i})_{i=\overline{1,n}}$. Note
that, throughout this paper, the indices $a$, $b$, $c$, $...$ run from 1 to $%
m$, while the indices $i,$ $j,$ $k,$ $...$ run from 1 to $n$. The Einstein
convention of summation is also adopted all over this work.

Let $J^{1\ast }(\mathcal{T},M)$ be the dual $1$-jet fibre bundle, whose
coordinates $(t^{a},x^{i},p_{i}^{a})$ are induced from $\mathcal{T}$ and $M$%
. The coordinate transformations from the product manifold $\mathcal{T}%
\times $ $M$ produce on $J^{1\ast }(\mathcal{T},M)$ the following coordinate
transformations:%
\begin{equation*}
\begin{array}{ccc}
\widetilde{t}^{a}=\widetilde{t}^{a}\left( t^{b}\right) , & \widetilde{x}^{i}=%
\widetilde{x}^{i}\left( x^{j}\right) , & \widetilde{p}_{i}^{a}=\dfrac{%
\partial x^{j}}{\partial \widetilde{x}^{i}}\dfrac{\partial \widetilde{t}^{a}%
}{\partial t^{b}}p_{j}^{b},%
\end{array}%
\end{equation*}%
where $\det \left( \partial \widetilde{t}^{a}/\partial t^{b}\right) \neq 0$
and $\det \left( \partial \widetilde{x}^{i}/\partial x^{j}\right) \neq 0.$

\begin{definition}
A pair of local functions on $E^{\ast }=J^{1\ast }(\mathcal{T},M),$ denoted
by 
\begin{equation*}
N=\left( \underset{1}{N}\overset{\left( a\right) }{_{\left( i\right) b}},\ 
\underset{2}{N}\overset{\left( a\right) }{_{\left( i\right) j}}\right) ,
\end{equation*}%
whose local components obey the transformation rules%
\begin{equation*}
\underset{1}{\widetilde{N}}\overset{\left( b\right) }{_{\left( j\right) c}}%
\dfrac{\delta \widetilde{t}^{c}}{\delta t^{a}}=\underset{1}{N}\overset{%
\left( c\right) }{_{\left( k\right) a}}\dfrac{\delta \widetilde{t}^{b}}{%
\delta t^{c}}\dfrac{\partial x^{k}}{\partial \widetilde{x}^{j}}-\dfrac{%
\partial \widetilde{p}_{j}^{b}}{\partial t^{a}},
\end{equation*}%
\begin{equation*}
\underset{2}{\widetilde{N}}\overset{\left( b\right) }{_{\left( j\right) k}}%
\dfrac{\partial \widetilde{x}^{k}}{\partial x^{i}}=\underset{2}{N}\overset{%
\left( c\right) }{_{\left( k\right) i}}\dfrac{\delta \widetilde{t}^{b}}{%
\delta t^{c}}\dfrac{\partial x^{k}}{\partial \widetilde{x}^{j}}-\dfrac{%
\partial \widetilde{p}_{j}^{b}}{\partial x^{i}},
\end{equation*}%
is called a \textbf{nonlinear connection} on $E^{\ast }$. The components $%
\underset{1}{N}\underset{}{\overset{\left( a\right) }{_{\left( i\right) b}}}$
(resp. $\underset{2}{N}\underset{}{\overset{\left( a\right) }{_{\left(
i\right) j}}}$) are called the\ \textbf{temporal} (resp. \textbf{spatial}) 
\textbf{components} of $N$.
\end{definition}

\begin{example}
Let $h_{ab}\left( t^{f}\right) $ (resp. $\varphi _{ij}\left( x^{k}\right) $)
be a semi-Riemannian metric on the temporal manifold $\mathcal{T}$ (resp.
spatial manifold $M$). Taking into account the local transformation rules of
the Christoffel symbols $\chi _{bc}^{a}\left( t\right) $ (resp. $\Gamma
_{ij}^{k}\left( x\right) $) of the metrics $h_{ab}\left( t\right) $ (resp. $%
\varphi _{ij}\left( x\right) $), then the pair of local functions%
\begin{equation*}
N_{0}=\left( \underset{1}{\overset{0}{N}}\overset{\left( a\right) }{_{\left(
i\right) b}},\ \underset{2}{\overset{0}{N}}\overset{\left( a\right) }{%
_{\left( i\right) j}}\right) ,
\end{equation*}%
where%
\begin{equation*}
\begin{array}{cc}
\underset{1}{\overset{0}{N}}\overset{\left( a\right) }{_{\left( i\right) b}}%
=\chi _{bc}^{a}p_{i}^{c},\quad & \underset{2}{\overset{0}{N}}\overset{\left(
a\right) }{_{\left( i\right) j}}\underset{}{}=-\Gamma _{ij}^{k}p_{k}^{a},%
\end{array}%
\end{equation*}%
represents a nonlinear connection on $E^{\ast }$. This is called the \textbf{%
canonical nonli-}\newline
\textbf{near connection attached to the metrics }$h_{ab}(t)$ \textbf{and} $%
\varphi _{ij}(x)$.
\end{example}

In what follows, we fix a nonlinear connection on $E^{\ast }$, and we
consider the \textit{adapted bases} of the nonlinear connection $N$, defined
by%
\begin{equation}
\left\{ \frac{\delta }{\delta t^{a}},\frac{\delta }{\delta x^{i}},\frac{%
\partial }{\partial p_{i}^{a}}\right\} \subset \mathcal{X}\left( E^{\ast
}\right) ,\quad \left\{ dt^{a},dx^{i},\delta p_{i}^{a}\right\} \subset 
\mathcal{X}^{\ast }\left( E^{\ast }\right) ,  \label{bz}
\end{equation}%
where%
\begin{equation*}
\begin{array}{l}
\dfrac{\delta }{\delta t^{a}}=\dfrac{\partial }{\partial t^{a}}-\underset{1}{%
N}\overset{\left( b\right) }{_{\left( j\right) a}}\dfrac{\partial }{\partial
p_{j}^{b}},\medskip \\ 
\dfrac{\delta }{\delta x^{i}}=\dfrac{\partial }{\partial x^{i}}-\underset{2}{%
N}\overset{\left( b\right) }{_{\left( j\right) i}}\dfrac{\partial }{\partial
p_{j}^{b}},\medskip \\ 
\delta p_{i}^{a}=dp_{i}^{a}+\underset{1}{N}\overset{\left( a\right) }{%
_{\left( i\right) b}}dt^{b}+\underset{2}{N}\overset{\left( a\right) }{%
_{\left( i\right) j}}dx^{j}.%
\end{array}%
\end{equation*}%
It is important to note that the transformation rules of the elements of the
adapted bases (\ref{bz}) are tensorial ones:%
\begin{equation}
\begin{array}{lll}
\dfrac{\delta }{\delta t^{a}}=\dfrac{\partial \tilde{t}^{b}}{\partial t^{a}}%
\dfrac{\delta }{\delta \tilde{t}^{b}},\medskip & \dfrac{\delta }{\delta x^{i}%
}=\dfrac{\partial \tilde{x}^{j}}{\partial x^{i}}\dfrac{\delta }{\delta 
\tilde{x}^{j}}, & \dfrac{\partial }{\partial p_{i}^{a}}=\dfrac{\partial 
\tilde{t}^{b}}{\partial t^{a}}\dfrac{\partial x^{i}}{\partial \tilde{x}^{j}}%
\dfrac{\partial }{\partial \tilde{p}_{j}^{b}}, \\ 
dt^{a}=\dfrac{\partial t^{a}}{\partial \tilde{t}^{b}}d\tilde{t}^{b}, & 
dx^{i}=\dfrac{\partial x^{i}}{\partial \tilde{x}^{j}}d\tilde{x}^{j}, & 
\delta p_{i}^{a}=\dfrac{\partial t^{a}}{\partial \tilde{t}^{b}}\dfrac{%
\partial \tilde{x}^{j}}{\partial x^{i}}\delta \tilde{p}_{j}^{b}.%
\end{array}
\label{schbz}
\end{equation}

\begin{remark}
The simple tensorial transformation rules (\ref{schbz}) of the adapted bases
(\ref{bz}) determined us to describe in what follows all geometrical objects
on the dual $1$-jet space $J^{1\ast }(\mathcal{T},M)$ in adapted local
components.
\end{remark}

In order to develop the geometrical theory of $N$-linear connections on the
dual $1$-jet space $E^{\ast }$, we need the following result:

\begin{proposition}
\emph{(i)} The Lie algebra $\mathcal{X}\left( E^{\ast }\right) $ of vector
fields decomposes as%
\begin{equation*}
\mathcal{X}\left( E^{\ast }\right) =\mathcal{X}\left( \mathcal{H}_{\mathcal{T%
}}\right) \oplus \mathcal{X}\left( \mathcal{H}_{M}\right) \oplus \mathcal{X}%
\left( \mathcal{V}\right) ,
\end{equation*}%
where%
\begin{equation*}
\mathcal{X}\left( \mathcal{H}_{\mathcal{T}}\right) {\scriptsize =}\text{%
\emph{Span}}\left\{ \dfrac{\delta }{\delta t^{a}}\right\} {\scriptsize %
,\quad }\mathcal{X}\left( \mathcal{H}_{M}\right) {\scriptsize =}\text{\emph{%
Span}}\left\{ \dfrac{\delta }{\delta x^{i}}\right\} {\scriptsize ,\quad }%
\mathcal{X}\left( \mathcal{V}\right) {\scriptsize =}\text{\emph{Span}}%
\left\{ \dfrac{\partial }{\partial p_{i}^{a}}\right\} {\scriptsize .}
\end{equation*}%
\emph{(ii)} The Lie algebra $\mathcal{X}^{\ast }\left( E^{\ast }\right) $ of
covector fields decomposes as%
\begin{equation*}
\mathcal{X}^{\ast }\left( E^{\ast }\right) =\mathcal{X}^{\ast }\left( 
\mathcal{H}_{\mathcal{T}}\right) \oplus \mathcal{X}^{\ast }\left( \mathcal{H}%
_{M}\right) \oplus \mathcal{X}^{\ast }\left( \mathcal{V}\right) ,
\end{equation*}%
where%
\begin{equation*}
\mathcal{X}^{\ast }\left( \mathcal{H}_{\mathcal{T}}\right) {\scriptsize =}%
\text{\emph{Span}}\left\{ dt^{a}\right\} {\scriptsize ,\quad }\mathcal{X}%
^{\ast }\left( \mathcal{H}_{M}\right) {\scriptsize =}\text{\emph{Span}}%
\left\{ dx^{i}\right\} {\scriptsize ,\quad }\mathcal{X}^{\ast }\left( 
\mathcal{V}\right) {\scriptsize =}\text{\emph{Span}}\left\{ \delta
p_{i}^{a}\right\} {\scriptsize .}
\end{equation*}
\end{proposition}

Let us consider that $h_{\mathcal{T}}$, $h_{M}$ (horizontal) and $v$
(vertical) are the canonical projections of the above decompositions. In
this context, we introduce the following geometrical concept:

\begin{definition}
A linear connection $D:\mathcal{X}\left( E^{\ast }\right) \times \mathcal{X}%
\left( E^{\ast }\right) \rightarrow \mathcal{X}\left( E^{\ast }\right) $ is
called \textit{an }$N$\textbf{-linear connection} on $E^{\ast }$ if and only
if $Dh_{\mathcal{T}}=0,Dh_{M}=0$ and $Dv=0$.
\end{definition}

It is obvious that the local description of the $N$-linear connection $D$ on 
$E^{\ast }$ is accomplished by \textit{nine} unique adapted components%
\begin{equation}
\begin{array}{c}
D\Gamma \left( N\right) =\left( A_{bc}^{a},\text{ }A_{jc}^{i},\text{ }%
-A_{\left( i\right) \left( b\right) c}^{\left( a\right) \left( j\right) },%
\text{ }H_{bk}^{a},\text{ }H_{jk}^{i},\text{ }-H_{\left( i\right) \left(
b\right) k}^{\left( a\right) \left( j\right) },\right. \medskip \\ 
\left. C_{b\left( c\right) }^{a\left( k\right) },\text{ }C_{j\left( c\right)
}^{i\left( k\right) },\text{ }-C_{\left( i\right) \left( b\right) \left(
c\right) }^{\left( a\right) \left( j\right) \left( k\right) }\right) ,%
\end{array}
\label{coef_N-cl}
\end{equation}%
which are locally defined by the relations:%
\begin{eqnarray*}
{\scriptsize D}_{\dfrac{\delta }{\delta t^{c}}}\frac{\delta }{\delta t^{b}} &%
{\scriptsize =}&{\scriptsize A}_{bc}^{a}\frac{\delta }{\delta t^{a}}%
{\scriptsize ,}\text{ }{\scriptsize D}_{\dfrac{\delta }{\delta t^{c}}}\frac{%
\delta }{\delta x^{j}}{\scriptsize =A}_{jc}^{i}\frac{\delta }{\delta x^{i}}%
{\scriptsize ,}\text{ }{\scriptsize D}_{\dfrac{\delta }{\delta t^{c}}}\frac{%
\partial }{\partial p_{j}^{b}}{\scriptsize =-A}_{\left( i\right) \left(
b\right) c}^{\left( a\right) \left( j\right) }\frac{\partial }{\partial
p_{i}^{a}}, \\
{\scriptsize D}_{\dfrac{\delta }{\delta x^{k}}}\frac{\delta }{\delta t^{b}} &%
{\scriptsize =}&{\scriptsize H}_{bk}^{a}\frac{\delta }{\delta t^{a}}%
{\scriptsize ,}\text{ }{\scriptsize D}_{\dfrac{\delta }{\delta x^{k}}}\frac{%
\delta }{\delta x^{j}}{\scriptsize =H}_{jk}^{i}\frac{\delta }{\delta x^{i}}%
{\scriptsize ,}\text{ }{\scriptsize D}_{\dfrac{\delta }{\delta x^{k}}}\frac{%
\partial }{\partial p_{j}^{b}}{\scriptsize =-H}_{\left( i\right) \left(
b\right) k}^{\left( a\right) \left( j\right) }\frac{\partial }{\partial
p_{i}^{a}}, \\
{\scriptsize D}_{\dfrac{\partial }{\partial p_{k}^{c}}}\frac{\delta }{\delta
t^{b}} &{\scriptsize =}&{\scriptsize C}_{b\left( c\right) }^{a\left(
k\right) }\frac{\delta }{\delta t^{a}}{\scriptsize ,}\text{ }{\scriptsize D}%
_{\dfrac{\partial }{\partial p_{k}^{c}}}\frac{\delta }{\delta x^{j}}%
{\scriptsize =C}_{j\left( c\right) }^{i\left( k\right) }\frac{\delta }{%
\delta x^{i}}{\scriptsize ,}\text{ }{\scriptsize D}_{\dfrac{\partial }{%
\partial p_{k}^{c}}}\frac{\partial }{\partial p_{j}^{b}}{\scriptsize =-C}%
_{\left( i\right) \left( b\right) \left( c\right) }^{\left( a\right) \left(
j\right) \left( k\right) }\frac{\partial }{\partial p_{i}^{a}}.
\end{eqnarray*}

\begin{example}
Let $N_{0}=\left( \underset{1}{\overset{0}{N}}\overset{\left( a\right) }{%
_{\left( i\right) b}},\ \underset{2}{\overset{0}{N}}\overset{\left( a\right) 
}{_{\left( i\right) j}}\right) $ be the canonical nonlinear connection
produced by the semi-Riemannian metrics $(h_{ab},\varphi _{ij})$. Taking
into account the transformation rules of the Christoffel symbols $\chi
_{bc}^{a}$ and $\Gamma _{jk}^{i}$, by local computations, we can show that
the local components%
\begin{equation*}
B\Gamma \left( N_{0}\right) =\left( \chi _{bc}^{a},\text{ }0,\text{ }%
-A_{\left( i\right) \left( b\right) c}^{\left( a\right) \left( j\right) },%
\text{ }0,\text{ }\Gamma _{jk}^{i},\text{ }-H_{\left( i\right) \left(
b\right) k}^{\left( a\right) \left( j\right) },\text{ }0,\text{ }0,\text{ }%
0\right)
\end{equation*}%
where%
\begin{equation*}
A_{\left( i\right) \left( b\right) c}^{\left( a\right) \left( j\right)
}=-\delta _{i}^{j}\chi _{bc}^{a},\quad H_{\left( i\right) \left( b\right)
k}^{\left( a\right) \left( j\right) }=\delta _{b}^{a}\Gamma _{ik}^{j},
\end{equation*}%
verify the transformation rules of the components of an $N$-linear
connection (for more details, see \emph{\cite{Atan+Neag1}}). Consequently, $%
B\Gamma \left( N_{0}\right) $ is an $N_{0}$-linear connection on $E^{\ast }$%
, which is called the \textbf{Berwald connection}\textit{\ of the metric
pair }$\left( h_{ab},\varphi _{ij}\right) .$
\end{example}

Now, let $D\Gamma (N)$ be an $N$-linear connection on $E^{\ast }$, locally
defined by (\ref{coef_N-cl}). The linear connection $D\Gamma (N)$ induces a
linear connection on the set of d-tensors on the dual $1$-jet fibre bundle $%
E^{\ast }=J^{1\ast }\left( \mathcal{T},M\right) ,$ in a natural way. Thus,
starting with a d-vector field $X$ and a d-tensor field $T$, locally
expressed by%
\begin{equation*}
\begin{array}{lll}
X & = & X^{a}\dfrac{\delta }{\delta t^{a}}+X^{i}\dfrac{\delta }{\delta x^{i}}%
+X_{\left( i\right) }^{\left( a\right) }\dfrac{\partial }{\partial p_{i}^{a}}%
,\medskip \\ 
T & = & T_{cj\left( b\right) \left( l\right) ...}^{ai\left( k\right) \left(
d\right) ...}\dfrac{\delta }{\delta t^{a}}\otimes \dfrac{\delta }{\delta
x^{i}}\otimes \dfrac{\partial }{\partial p_{l}^{d}}\otimes dt^{c}\otimes
dx^{j}\otimes \delta p_{k}^{b}\otimes ...,%
\end{array}%
\end{equation*}%
we can define the covariant derivative%
\begin{equation*}
\begin{array}{l}
D_{X}T=X^{g}D_{\dfrac{\delta }{\delta t^{g}}}T+X^{s}D_{\dfrac{\delta }{%
\delta x^{s}}}T+X_{\left( s\right) }^{\left( g\right) }D_{\dfrac{\partial }{%
\partial p_{s}^{g}}}T=\medskip \\ 
=\left\{ X^{g}T_{cj\left( b\right) \left( l\right) .../g}^{ai\left( k\right)
\left( d\right) ...}+X^{s}T_{cj\left( b\right) \left( l\right)
...|s}^{ai\left( k\right) \left( d\right) ...}+\right. \medskip \\ 
\left. +X_{\left( s\right) }^{\left( g\right) }T_{cj\left( b\right) \left(
l\right) ...}^{ai\left( k\right) \left( d\right) ...}\mid _{\left( g\right)
}^{\left( s\right) }\right\} \dfrac{\delta }{\delta t^{a}}\otimes \dfrac{%
\delta }{\delta x^{i}}\otimes \dfrac{\partial }{\partial p_{l}^{d}}\otimes
dt^{c}\otimes dx^{j}\otimes \delta p_{k}^{b}\otimes ...,%
\end{array}%
\end{equation*}%
where

\begin{itemize}
\item the $\mathcal{T}$\textit{-horizontal covariant derivative }of $D\Gamma
(N)$:%
\begin{equation*}
\left( h_{\mathcal{T}}\right) \text{ }\left\{ 
\begin{array}{l}
T_{cj\left( b\right) \left( l\right) .../g}^{ai\left( k\right) \left(
d\right) ...}=\dfrac{\delta T_{cj\left( b\right) \left( l\right)
...}^{ai\left( k\right) \left( d\right) ...}}{\delta t^{g}}+T_{cj\left(
b\right) \left( l\right) ...}^{fi\left( k\right) \left( d\right)
...}A_{fg}^{a}+\medskip \\ 
+T_{cj\left( b\right) \left( l\right) ...}^{ar\left( k\right) \left(
d\right) ...}A_{rg}^{i}+T_{cj\left( f\right) \left( l\right) ...}^{ai\left(
r\right) \left( d\right) ...}A_{\left( r\right) \left( b\right) g}^{\left(
f\right) \left( k\right) }+...-\medskip \\ 
-T_{fj\left( b\right) \left( l\right) ...}^{ai\left( k\right) \left(
d\right) ...}A_{cg}^{f}-T_{cr\left( b\right) \left( l\right) ...}^{ai\left(
k\right) \left( d\right) ...}A_{jg}^{r}-T_{cj\left( b\right) \left( r\right)
...}^{ai\left( k\right) \left( f\right) ...}A_{\left( l\right) \left(
f\right) g}^{\left( d\right) \left( r\right) }-...,%
\end{array}%
\right.
\end{equation*}

\item the $M$\textit{-horizontal covariant derivative }of $D\Gamma (N)$:%
\begin{equation*}
\left( h_{M}\right) \text{ }\left\{ 
\begin{array}{l}
T_{cj\left( b\right) \left( l\right) ...|s}^{ai\left( k\right) \left(
d\right) ...}=\dfrac{\delta T_{cj\left( b\right) \left( l\right)
...}^{ai\left( k\right) \left( d\right) ...}}{\delta x^{s}}+T_{cj\left(
b\right) \left( l\right) ...}^{fi\left( k\right) \left( d\right)
...}H_{fs}^{a}+\medskip \\ 
+T_{cj\left( b\right) \left( l\right) ...}^{ar\left( k\right) \left(
d\right) ...}H_{rs}^{i}+T_{cj\left( f\right) \left( l\right) ...}^{ai\left(
r\right) \left( d\right) ...}H_{\left( r\right) \left( b\right) s}^{\left(
f\right) \left( k\right) }+...-\medskip \\ 
-T_{fj\left( b\right) \left( l\right) ...}^{ai\left( k\right) \left(
d\right) ...}H_{cs}^{f}-T_{cr\left( b\right) \left( l\right) ...}^{ai\left(
k\right) \left( d\right) ...}H_{js}^{r}-T_{cj\left( b\right) \left( r\right)
...}^{ai\left( k\right) \left( f\right) ...}H_{\left( l\right) \left(
f\right) s}^{\left( d\right) \left( r\right) }-...,%
\end{array}%
\right.
\end{equation*}

\item the \textit{vertical covariant derivative }of $D\Gamma (N)$:%
\begin{equation*}
\left( v\right) \text{ }\left\{ 
\begin{array}{l}
T_{cj\left( b\right) \left( l\right) ...}^{ai\left( k\right) \left( d\right)
...}|_{\left( g\right) }^{\left( s\right) }=\dfrac{\partial T_{cj\left(
b\right) \left( l\right) ...}^{ai\left( k\right) \left( d\right) ...}}{%
\partial p_{s}^{g}}+T_{cj\left( b\right) \left( l\right) ...}^{fi\left(
k\right) \left( d\right) ...}C_{f\left( g\right) }^{a\left( s\right)
}+\medskip \\ 
+T_{cj\left( b\right) \left( l\right) ...}^{ar\left( k\right) \left(
d\right) ...}C_{r\left( g\right) }^{i\left( s\right) }+T_{cj\left( f\right)
\left( l\right) ...}^{ai\left( r\right) \left( d\right) ...}C_{\left(
r\right) \left( b\right) \left( g\right) }^{\left( f\right) \left( k\right)
\left( s\right) }+...-\medskip \\ 
-T_{fj\left( b\right) \left( l\right) ...}^{ai\left( k\right) \left(
d\right) ...}C_{c\left( g\right) }^{f\left( s\right) }-T_{cr\left( b\right)
\left( l\right) ...}^{ai\left( k\right) \left( d\right) ...}C_{j\left(
g\right) }^{r\left( s\right) }-T_{cj\left( b\right) \left( r\right)
...}^{ai\left( k\right) \left( f\right) ...}C_{\left( l\right) \left(
f\right) \left( g\right) }^{\left( d\right) \left( r\right) \left( s\right)
}-...\text{.}%
\end{array}%
\right.
\end{equation*}
\end{itemize}

\begin{remark}
If $T=Y$ is a d-vector field on $E^{\ast },$ locally expressed by%
\begin{equation*}
Y=Y^{a}\frac{\delta }{\delta t^{a}}+Y^{i}\frac{\delta }{\delta x^{i}}%
+Y_{\left( i\right) }^{\left( a\right) }\frac{\partial }{\partial p_{i}^{a}},
\end{equation*}%
then the following expressions of the local covariant derivatives hold good:%
\begin{equation*}
\left( h_{\mathcal{T}}\right) \text{ }\left\{ 
\begin{array}{l}
Y_{\text{ }/c}^{a}=\dfrac{\delta Y^{a}}{\delta t^{c}}+Y^{b}A_{bc}^{a},%
\medskip \\ 
Y_{\text{ }/c}^{i}=\dfrac{\delta Y^{i}}{\delta t^{c}}+Y^{j}A_{jc}^{i},%
\medskip \\ 
Y_{\left( i\right) /c}^{\left( a\right) }=\dfrac{\delta Y_{\left( i\right)
}^{\left( a\right) }}{\delta t^{c}}-Y_{\left( j\right) }^{\left( b\right)
}A_{\left( i\right) \left( b\right) c}^{\left( a\right) \left( j\right) },%
\end{array}%
\right. \left( h_{M}\right) \text{ }\left\{ 
\begin{array}{l}
Y_{\text{ }|k}^{a}=\dfrac{\delta Y^{a}}{\delta x^{k}}+Y^{b}H_{bk}^{a},%
\medskip \\ 
Y_{\text{ }|k}^{i}=\dfrac{\delta Y^{i}}{\delta x^{k}}+Y^{j}H_{jk}^{i},%
\medskip \\ 
Y_{\left( i\right) |k}^{\left( a\right) }=\dfrac{\delta Y_{\left( i\right)
}^{\left( a\right) }}{\delta x^{k}}-Y_{\left( j\right) }^{\left( b\right)
}H_{\left( i\right) \left( b\right) k}^{\left( a\right) \left( j\right) },%
\end{array}%
\right.
\end{equation*}
\end{remark}

\begin{equation*}
\left( v\right) \text{ }\left\{ 
\begin{array}{l}
Y^{a}|_{\left( c\right) }^{\left( k\right) }=\dfrac{\partial Y^{a}}{\partial
p_{k}^{c}}+Y^{b}C_{b\left( c\right) }^{a\left( k\right) },\medskip \\ 
Y^{i}|_{\left( c\right) }^{\left( k\right) }=\dfrac{\partial Y^{i}}{\partial
p_{k}^{c}}+Y^{j}C_{j\left( c\right) }^{i\left( k\right) },\medskip \\ 
Y_{\left( i\right) }^{\left( a\right) }|_{\left( c\right) }^{\left( k\right)
}=\dfrac{\partial Y_{\left( i\right) }^{\left( a\right) }}{\partial p_{k}^{c}%
}-Y_{\left( j\right) }^{\left( b\right) }C_{\left( i\right) \left( b\right)
\left( c\right) }^{\left( a\right) \left( j\right) \left( k\right) }.%
\end{array}%
\right.
\end{equation*}

\section{Components of $h$-normal $N$-linear connections on dual $1$-jet
spaces}

Because the number of components which characterize an $N$-linear connection
on $E^{\ast }$ is big one (nine local components), we are constrained to
study only a particular class of $N$-linear connections on $E^{\ast }$,
which must be characterized by a reduced number of components. In this
direction, let us fix on the temporal manifold $\mathcal{T}$ a
semi-Riemannian metric $h_{ab}$, together with its Christoffel symbols $\chi
_{bc}^{a}$. Let $\mathbb{J}$ be the $h$\textit{-normalization d-tensor field 
}on $E^{\ast }$, locally expressed by \cite{Atan+Neag1}%
\begin{equation*}
\mathbb{J}=J_{\left( a\right) bj}^{\left( i\right) }\delta p_{i}^{a}\otimes
dt^{b}\otimes dx^{j},
\end{equation*}%
where $J_{\left( a\right) bj}^{\left( i\right) }=h_{ab}\delta _{j}^{i}$. In
this context, we introduce the following geometrical concept:

\begin{definition}
An $N$-linear connection $D\Gamma (N)$ on $E^{\ast }$, whose local
components (\ref{coef_N-cl}) verify the relations%
\begin{equation*}
\begin{array}{cccc}
A_{bc}^{a}=\chi _{bc}^{a},\quad & H_{bi}^{a}=0,\quad & C_{b\left( c\right)
}^{a\left( i\right) }=0,\quad & D\mathbb{J}=0,%
\end{array}%
\end{equation*}%
is called an $h$\textbf{-normal }$N$\textbf{-linear connection} on the dual $%
1$-jet fibre bundle $E^{\ast }$.
\end{definition}

\begin{theorem}
The adapted components of an $h$-normal $N$-linear connection $D\Gamma (N)$
verify the following identities:%
\begin{equation}
\begin{array}{c}
A_{bc}^{a}=\chi _{bc}^{a},\quad H_{bi}^{a}=0,\quad C_{b\left( c\right)
}^{a\left( i\right) }=0,\medskip \\ 
A_{\left( i\right) \left( b\right) c}^{\left( a\right) \left( j\right)
}=\delta _{b}^{a}A_{ic}^{j}-\delta _{i}^{j}\chi _{bc}^{a},\quad H_{\left(
i\right) \left( b\right) k}^{\left( a\right) \left( j\right) }=\delta
_{b}^{a}H_{ik}^{j},\medskip \\ 
C_{\left( i\right) \left( b\right) \left( c\right) }^{\left( a\right) \left(
j\right) \left( k\right) }=\delta _{b}^{a}C_{i\left( c\right) }^{j\left(
k\right) }.%
\end{array}
\label{2.22}
\end{equation}
\end{theorem}

\begin{proof}
It is obvious that the first three relations come immediately from the
definition of an $h$-normal $N$-linear connection. To prove the other three
relations, we emphasize that, taking into account the definition of the
local $\mathcal{T}$-horizontal\ ($"_{/g}"$), $M$-horizontal ($"_{|s}"$) and
vertical ($"|_{\left( g\right) }^{\left( s\right) }")$ covariant derivatives
produced by $D\Gamma (N)$, the condition $D\mathbb{J}=0$ is equivalent to%
\begin{equation*}
\begin{array}{ccc}
J_{\left( a\right) bj/g}^{\left( i\right) }=0,\quad & J_{\left( a\right)
bj|s}^{\left( i\right) }=0,\quad & J_{\left( a\right) bj}^{\left( i\right)
}|_{\left( g\right) }^{\left( s\right) }=0.%
\end{array}%
\end{equation*}%
Consequently, the condition $D\mathbb{J}=0$ provides the local identities%
\begin{equation*}
\begin{array}{c}
h_{bf}A_{\left( j\right) \left( a\right) c}^{\left( f\right) \left( i\right)
}=h_{ab}A_{jc}^{i}-\delta _{j}^{i}\left( \dfrac{\partial h_{ab}}{\partial
t^{c}}-h_{ag}\chi _{bc}^{g}\right) ,\medskip \\ 
h_{bf}H_{\left( j\right) \left( a\right) k}^{\left( f\right) \left( i\right)
}=h_{ba}H_{jk}^{i},\quad h_{bf}C_{\left( j\right) \left( a\right) \left(
c\right) }^{\left( f\right) \left( i\right) \left( k\right)
}=h_{ba}C_{j\left( c\right) }^{i\left( k\right) }.%
\end{array}%
\end{equation*}%
Contracting now the above relations by $h^{be}$, we obtain the last required
identities from (\ref{2.22}).
\end{proof}

\begin{remark}
The above theorem says us that an $h$-normal $N$-linear connection on $%
E^{\ast }$ is an $N$-linear connection determined by \textbf{four} effective
components (instead of nine in the general case):%
\begin{equation*}
D\Gamma (N)=\left( \chi _{bc}^{a},\text{ }A_{jc}^{i},\text{ }H_{jk}^{i},%
\text{ }C_{j\left( c\right) }^{i\left( k\right) }\right) .
\end{equation*}%
The other five components either vanish or are provided by the relations (%
\ref{2.22}). Consequently, we can assert that the Berwald $N_{0}$-linear
connection associated to the pair of metrics $\left( h_{ab},\varphi
_{ij}\right) $ is an $h$-normal $N_{0}$-linear connection on $E^{\ast }$,
whose four effective components are%
\begin{equation*}
B\Gamma \left( N_{0}\right) =\left( \chi _{bc}^{a},\text{ }0,\text{ }\Gamma
_{jk}^{i},\text{ }0\right) .
\end{equation*}
\end{remark}

\section{Adapted components of torsion and curvature tensors}

The study of the adapted components of the torsion and curvature tensors of
an arbitrary $N$-linear connection $D\Gamma (N)$ on $E^{\ast }$ was done in 
\cite{Atan+Neag1}. In that context, one proves that the torsion tensor $%
\mathbb{T}$ is determined by \textit{twelve} effective local adapted
d-tensors, while the curvature tensor $\mathbb{R}$ is determined by \textit{%
eighteen} local adapted d-tensors. In what follows, we study the adapted
components of the torsion and curvature tensors for an $h$-normal $N$-linear
connection $D\Gamma (N)$.

\begin{theorem}
The torsion tensor $\mathbb{T}$ of an $h$-normal $N$-linear connection $%
D\Gamma (N)$ is determined by \textbf{nine} effective local adapted
d-tensors (instead of twelve in the general case):%
\begin{equation}
\begin{tabular}{||l||l||l||l||}
\hline\hline
& $h_{T}$ & $h_{M}$ & $v$ \\ \hline\hline
$h_{T}h_{T}$ & $0$ & $0$ & $R_{\left( r\right) ab}^{\left( f\right) }$ \\ 
\hline\hline
$h_{M}h_{T}$ & $0$ & $T_{aj}^{r}$ & $R_{\left( r\right) aj}^{\left( f\right)
}$ \\ \hline\hline
$vh_{T}$ & $0$ & $0$ & $P_{\left( r\right) a\left( b\right) }^{\left(
f\right) \left. {}\right. \left( j\right) }$ \\ \hline\hline
$h_{M}h_{M}$ & $0$ & $T_{ij}^{r}$ & $R_{\left( r\right) ij}^{\left( f\right)
}$ \\ \hline\hline
$vh_{M}$ & $0$ & $P_{i\left( b\right) }^{r\left( j\right) }$ & $P_{\left(
r\right) i\left( b\right) }^{\left( f\right) \left. {}\right. \left(
j\right) }$ \\ \hline\hline
$vv$ & $0$ & $0$ & $S_{\left( r\right) \left( a\right) \left( b\right)
}^{\left( f\right) \left( i\right) \left( j\right) }$ \\ \hline\hline
\end{tabular}
\label{Table1}
\end{equation}%
where%
\begin{equation*}
T_{aj}^{r}=-A_{aj}^{r},\quad T_{ij}^{r}=H_{ij}^{r}-H_{ji}^{r},\quad
P_{i\left( b\right) }^{r\left( j\right) }=C_{i\left( b\right) }^{r\left(
j\right) },
\end{equation*}%
\begin{equation*}
P_{\left( r\right) a\left( b\right) }^{\left( f\right) \ \left( j\right) }=%
\dfrac{\partial \underset{1}{N}\overset{\left( f\right) }{_{\left( r\right)
a}}}{\partial p_{j}^{b}}+\delta _{b}^{f}A_{ra}^{j}-\delta _{r}^{j}\chi
_{ba}^{f},\quad P_{\left( r\right) i\left( b\right) }^{\left( f\right) \
\left( j\right) }=\dfrac{\partial \underset{2}{N}\overset{\left( f\right) }{%
_{\left( r\right) i}}}{\partial p_{j}^{b}}+\delta _{b}^{f}H_{ri}^{j},
\end{equation*}%
\begin{equation*}
R_{\left( r\right) ab}^{\left( f\right) }=\dfrac{\delta \underset{1}{N}%
\overset{\left( f\right) }{_{\left( r\right) a}}}{\delta t^{b}}-\dfrac{%
\delta \underset{1}{N}\overset{\left( f\right) }{_{\left( r\right) b}}}{%
\delta t^{a}},\quad R_{\left( r\right) aj}^{\left( f\right) }=\dfrac{\delta 
\underset{1}{N}\overset{\left( f\right) }{_{\left( r\right) a}}}{\delta x^{j}%
}-\dfrac{\delta \underset{2}{N}\overset{\left( f\right) }{_{\left( r\right)
i}}}{\delta t^{a}},
\end{equation*}%
\begin{equation*}
R_{\left( r\right) ij}^{\left( f\right) }=\dfrac{\delta \underset{2}{N}%
\overset{\left( f\right) }{_{\left( r\right) i}}}{\delta x^{j}}-\dfrac{%
\delta \underset{2}{N}\overset{\left( f\right) }{_{\left( r\right) j}}}{%
\delta x^{i}},\quad S_{\left( r\right) \left( a\right) \left( b\right)
}^{\left( f\right) \left( i\right) \left( j\right) }=-\left( \delta
_{a}^{f}C_{r\left( b\right) }^{i\left( j\right) }-\delta _{b}^{f}C_{r\left(
a\right) }^{j\left( i\right) }\right) .
\end{equation*}
\end{theorem}

\begin{proof}
Particularizing the general local expressions from \cite{Atan+Neag1}, which
generally give those twelve d-components of the torsion tensor of an $N$%
-linear connection, an $h$-normal $N$-linear connection $D\Gamma (N)$, we
deduce that the adapted components $T_{bc}^{a},$ $T_{bj}^{a}$ and $%
P_{b\left( c\right) }^{a\left( k\right) }$ vanish, while the other nine are
given by the formulas from theorem.
\end{proof}

\begin{remark}
All torsion d-tensors of the Berwald $h$-normal $N_{0}$-linear connection $%
B\Gamma \left( N_{0}\right) $ (associated to the metrics $h_{ab}$ and $%
\varphi _{ij}$) are zero, except%
\begin{equation*}
\begin{array}{ll}
R_{\left( r\right) ab}^{\left( f\right) }=\chi _{gab}^{f}p_{r}^{g}, & 
R_{\left( r\right) ij}^{\left( f\right) }=-\mathcal{R}_{rij}^{s}p_{s}^{f},%
\end{array}%
\end{equation*}%
where $\chi _{gab}^{f}(t)$ (resp. $\mathcal{R}_{rij}^{s}(x)$) are the local
curvature tensors of the semi-Rie\-ma\-nni\-an metric $h_{ab}$ (resp. $%
\varphi _{ij}$).
\end{remark}

\begin{theorem}
The curvature tensor $\mathbb{R}$ of an $h$-normal $N$-linear connection $%
D\Gamma (N)$ is characterized by \textbf{seven} effective adapted local
d-tensors (instead of eighteen in the general case):%
\begin{equation}
\begin{tabular}{||l||l||l||l||}
\hline\hline
& $h_{\mathcal{T}}$ & $h_{M}$ & $v$ \\ \hline\hline
$h_{\mathcal{T}}h_{\mathcal{T}}$ & $\chi _{abc}^{d}$ & $R_{ibc}^{l}$ & $%
-R_{\left( l\right) \left( a\right) bc}^{\left( d\right) \left( i\right)
}=\delta _{l}^{i}\chi _{abc}^{d}-\delta _{a}^{d}R_{lbc}^{i}$ \\ \hline\hline
$h_{M}h_{\mathcal{T}}$ & $0$ & $R_{ibk}^{l}$ & $-R_{\left( l\right) \left(
a\right) bk}^{\left( d\right) \left( i\right) }=-\delta _{a}^{d}R_{lbk}^{i}$
\\ \hline\hline
$wh_{\mathcal{T}}$ & $0$ & $P_{ib\left( c\right) }^{l\ \left( k\right) }$ & $%
-P_{\left( l\right) \left( a\right) b\left( c\right) }^{\left( d\right)
\left( i\right) \ \left( k\right) }=-\delta _{a}^{d}P_{lb\left( c\right)
}^{i\ \left( k\right) }$ \\ \hline\hline
$h_{M}h_{M}$ & $0$ & $R_{ijk}^{l}$ & $-R_{\left( l\right) \left( a\right)
jk}^{\left( d\right) \left( i\right) }=-\delta _{a}^{d}R_{ljk}^{i}$ \\ 
\hline\hline
$wh_{M}$ & $0$ & $P_{ij\left( c\right) }^{l\ \left( k\right) }$ & $%
-P_{\left( l\right) \left( a\right) j\left( c\right) }^{\left( d\right)
\left( i\right) \ \left( k\right) }=-\delta _{a}^{d}P_{lj\left( c\right)
}^{i\ \left( k\right) }$ \\ \hline\hline
$ww$ & $0$ & $S_{i\left( b\right) \left( c\right) }^{l\left( j\right) \left(
k\right) }$ & $-S_{\left( l\right) \left( a\right) \left( b\right) \left(
c\right) }^{\left( d\right) \left( i\right) \left( j\right) \left( k\right)
}=-\delta _{a}^{d}S_{l\left( b\right) \left( c\right) }^{i\left( j\right)
\left( k\right) }$ \\ \hline\hline
\end{tabular}
\label{Table2}
\end{equation}%
where$\medskip $

$R_{abc}^{d}:=\chi _{abc}^{d}=\dfrac{\delta \chi _{ab}^{d}}{\delta t^{c}}-%
\dfrac{\delta \chi _{ac}^{d}}{\delta t^{b}}+\chi _{ab}^{f}\chi
_{fc}^{d}-\chi _{ac}^{f}\chi _{fb}^{d},\medskip $

$R_{ibc}^{l}=\dfrac{\delta A_{ib}^{l}}{\delta t^{c}}-\dfrac{\delta A_{ic}^{l}%
}{\delta t^{b}}+A_{ib}^{r}A_{rc}^{l}-A_{ic}^{r}A_{rb}^{l}+C_{i\left(
f\right) }^{l\left( r\right) }R_{\left( r\right) bc}^{\left( f\right)
},\medskip $

$R_{ibk}^{l}=\dfrac{\delta A_{ib}^{l}}{\delta x^{k}}-\dfrac{\delta H_{ik}^{l}%
}{\delta t^{b}}+A_{ib}^{r}H_{rk}^{l}-H_{ik}^{r}A_{rb}^{l}+C_{i\left(
f\right) }^{l\left( r\right) }R_{\left( r\right) bk}^{\left( f\right)
},\medskip $

$P_{ib\left( c\right) }^{l\ \left( k\right) }=\dfrac{\partial A_{ib}^{l}}{%
\partial p_{k}^{c}}-C_{i\left( c\right) /b}^{l\left( k\right) }+C_{i\left(
f\right) }^{l\left( r\right) }P_{\left( r\right) b\left( c\right) }^{\left(
f\right) \ \left( k\right) },\medskip $

$R_{ijk}^{l}=\dfrac{\delta H_{ij}^{l}}{\delta x^{k}}-\dfrac{\delta H_{ik}^{l}%
}{\delta x^{j}}+H_{ij}^{r}H_{rk}^{l}-H_{ik}^{r}H_{rj}^{l}+C_{i\left(
f\right) }^{l\left( r\right) }R_{\left( r\right) jk}^{\left( f\right)
},\medskip $

$P_{ij\left( c\right) }^{l\ \left( k\right) }=\dfrac{\partial H_{ij}^{l}}{%
\partial p_{k}^{c}}-C_{i\left( c\right) |j}^{l\left( k\right) }+C_{i\left(
r\right) }^{l\left( f\right) }P_{\left( f\right) j\left( c\right) }^{\left(
r\right) \ \left( k\right) },\medskip $

$S_{i\left( b\right) \left( c\right) }^{l\left( j\right) \left( k\right) }=%
\dfrac{\partial C_{i\left( b\right) }^{l\left( j\right) }}{\partial p_{k}^{c}%
}-\dfrac{\partial C_{i\left( c\right) }^{l\left( k\right) }}{\partial
p_{j}^{b}}+C_{i\left( b\right) }^{r\left( j\right) }C_{r\left( c\right)
}^{l\left( k\right) }-C_{i\left( c\right) }^{r\left( k\right) }C_{r\left(
b\right) }^{l\left( j\right) }.$
\end{theorem}

\begin{proof}
The general formulas that express the local curvature d-tensors of an
arbitrary $N$-linear connection (for more details, see \cite{Atan+Neag1}),
applied to the particular case of an $h$-normal $N$-linear connection $%
D\Gamma (N)$, imply the above formulas and the relations from the Table (\ref%
{Table2}).
\end{proof}

\begin{remark}
In the case of the Berwald $h$-normal $N_{0}$-linear connection $B\Gamma
(N_{0})$ (associated to the pair of metrics $\left( h_{ab},\varphi
_{ij}\right) $), all curvature d-tensors are zero, except%
\begin{equation*}
R_{abc}^{d}=\chi _{abc}^{d},\quad R_{\left( l\right) \left( a\right)
bc}^{\left( d\right) \left( i\right) }=-\delta _{l}^{i}\chi _{abc}^{d},\quad
R_{ijk}^{l}=\mathcal{R}_{ijk}^{l}\quad R_{\left( i\right) \left( a\right)
jk}^{\left( d\right) \left( l\right) }=\delta _{a}^{d}\mathcal{R}_{ijk}^{l},
\end{equation*}%
where $\chi _{gab}^{f}(t)$ (resp. $\mathcal{R}_{rij}^{s}(x)$) are the local
curvature tensors of the semi-Rie\-ma\-nni\-an metric $h_{ab}$ (resp. $%
\varphi _{ij}$).
\end{remark}

\section{Local Ricci identities. Non-metrical deflection d-tensor identities}

Let us consider now the following more particular geometrical concept:

\begin{definition}
An $h$-normal $N$-linear connection, whose local components%
\begin{equation*}
CD\Gamma (N)=\left( \chi _{bc}^{a},\text{ }A_{jc}^{i},\text{ }H_{jk}^{i},%
\text{ }C_{j\left( c\right) }^{i\left( k\right) }\right) ,
\end{equation*}%
verify the relations%
\begin{equation*}
H_{jk}^{i}=H_{kj}^{i},\qquad C_{j\left( c\right) }^{i\left( k\right)
}=C_{j\left( c\right) }^{k\left( i\right) },
\end{equation*}%
is called an $h$\textbf{-normal }$N$\textbf{-linear connection of Cartan
type }or a $CD\Gamma (N)$\textbf{-linear connection} on $E^{\ast }=J^{1\ast
}\left( \mathcal{T},M\right) $.
\end{definition}

\begin{remark}
The torsion tensor $\mathbb{T}$ of an $h$-normal $N$-linear connection of
Cartan type $CD\Gamma (N)$ is characterized only by \textbf{eight} adapted
local d-tensors because the torsion components $%
T_{jk}^{i}=H_{jk}^{i}-H_{kj}^{i}$ from the Table (\ref{Table1}) are vanishing%
$.$
\end{remark}

\begin{example}
Taking into account that the Christoffel symbols $\Gamma _{jk}^{i}(x)$ of
the spatial metric $\varphi _{ij}(x)$ are symmetric, it follows that the
Berwald $h$-normal $N_{0}$-linear connection $B\Gamma (N_{0})$ is of Cartan
type.
\end{example}

\begin{theorem}
The following local \textbf{Ricci identities} for a $CD\Gamma (N)$-linear
connection are true:

\begin{itemize}
\item the $h_{\mathcal{T}}$-Ricci identities:\bigskip

$\medskip X_{/b/c}^{a}-X_{/c/b}^{a}=X^{f}\chi
_{fbc}^{a}-X^{a}|_{(f)}^{(r)}R_{(r)bc}^{(f)},$

$\medskip
X_{/b|k}^{a}-X_{|k/b}^{a}=-X_{|r}^{a}T_{bk}^{r}-X^{a}|_{(f)}^{(r)}R_{(r)bk}^{(f)}, 
$

$\medskip X_{|j|k}^{a}-X_{|k|j}^{a}=-X^{a}|_{(f)}^{(r)}R_{(r)jk}^{(f)},$

$\medskip
X_{/b}^{a}|_{(c)}^{(k)}-X^{a}|_{(c)/b}^{(k)}=-X^{a}|_{(f)}^{(r)}P_{(r)b(c)}^{(f)\;\;(k)}, 
$

$\medskip
X_{|j}^{a}|_{(c)}^{(k)}-X^{a}|_{(c)|j}^{(k)}=-X_{|r}^{a}C_{j(c)}^{r(k)}-X^{a}|_{(f)}^{(r)}P_{(r)j(c)}^{(f)\;(k)}, 
$

$\bigskip
X^{a}|_{(b)}^{(j)}|_{(c)}^{(k)}-X^{a}|_{(c)}^{(k)}|_{(b)}^{(j)}=-X^{a}|_{(f)}^{(r)}S_{(r)(b)(c)}^{(f)(j)(k)}; 
$

\item the $h_{M}$-Ricci identities:$\bigskip$

$\medskip
X_{/b/c}^{i}-X_{/c/b}^{i}=X^{r}R_{rbc}^{i}-X^{i}|_{(f)}^{(r)}R_{(r)bc}^{(f)}, 
$

$\medskip
X_{/b|k}^{i}-X_{|k/b}^{i}=X^{r}R_{rbk}^{i}-X_{|r}^{i}T_{bk}^{r}-X^{i}|_{(f)}^{(r)}R_{(r)bk}^{(f)}, 
$

$\medskip
X_{|j|k}^{i}-X_{|k|j}^{i}=X^{r}R_{rjk}^{i}-X^{i}|_{(f)}^{(r)}R_{(r)jk}^{(f)}, 
$

$\medskip
X_{/b}^{i}|_{(c)}^{(k)}-X^{i}|_{(c)/b}^{(k)}=X^{r}P_{rb(c)}^{i\;%
\;(k)}-X^{i}|_{(f)}^{(r)}P_{(r)b(c)}^{(f)\;\;(k)},$

$\medskip X_{|j}^{i}|_{(c)}^{(k)}-X^{i}|_{(c)|j}^{(k)}=X^{r}P_{rj(c)}^{i\;\
(k)}-X_{|r}^{i}C_{j(c)}^{r(k)}-X^{i}|_{(f)}^{(r)}P_{(r)j(c)}^{(f)\;(k)},$

$%
X^{i}|_{(b)}^{(j)}|_{(c)}^{(k)}-X^{i}|_{(c)}^{(k)}|_{(b)}^{(j)}=X^{r}S_{r(b)(c)}^{i(j)(k)}-X^{i}|_{(f)}^{(r)}S_{(r)(b)(c)}^{(f)(j)(k)};\bigskip
$

\item the $v$-Ricci identities:$\bigskip$

$\medskip
X_{(i)/b/c}^{(a)}-X_{(i)/c/b}^{(a)}=X_{(r)}^{(a)}R_{ibc}^{r}-X_{(i)}^{(f)}%
\chi _{fbc}^{a}-X_{(i)}^{(a)}|_{(f)}^{(r)}R_{(r)bc}^{(f)},$

$\medskip
X_{(i)/b|k}^{(a)}-X_{(i)|k/b}^{(a)}=X_{(r)}^{(a)}R_{ibk}^{r}-X_{(i)|r}^{(a)}T_{bk}^{r}-X_{(i)}^{(a)}|_{(f)}^{(r)}R_{(r)bk}^{(f)}, 
$

$\medskip
X_{(i)|j|k}^{(a)}-X_{(i)|k|j}^{(a)}=X_{(r)}^{(a)}R_{ijk}^{r}-X_{(i)}^{(a)}|_{(f)}^{(r)}R_{(r)jk}^{(f)}, 
$

$\medskip
X_{(i)/b}^{(a)}|_{(c)}^{(k)}-X_{(i)}^{(a)}|_{(c)/b}^{(k)}=X_{(r)}^{(a)}P_{ib(c)}^{r\;\;(k)}-X_{(i)}^{(a)}|_{(f)}^{(r)}P_{(r)b(c)}^{(f)\;\;(k)}, 
$

$\medskip
X_{(i)|j}^{(a)}|_{(c)}^{(k)}-X_{(i)}^{(a)}|_{(c)|j}^{(k)}=X_{(r)}^{(a)}P_{ij(c)}^{r\;\ (k)}-X_{(i)|r}^{(a)}C_{j(c)}^{r(k)}-X_{(i)}^{(a)}|_{(f)}^{(r)}P_{(r)j(c),}^{(f)\;(k)} 
$

$\bigskip
X_{(i)}^{(a)}|_{(b)}^{(j)}|_{(c)}^{(k)}-X_{(i)}^{(a)}|_{(c)}^{(k)}|_{(b)}^{(j)}=X_{(r)}^{(a)}S_{i(b)(c)}^{r(j)(k)}-X_{(i)}^{(a)}|_{(f)}^{(r)}S_{(r)(b)(c)}^{(f)(j)(k)},%
\newline
$where%
\begin{equation*}
{X=X^{a}{\dfrac{\delta }{\delta t^{a}}}+X^{i}{\dfrac{\delta }{\delta x^{i}}}%
+X_{(i)}^{(a)}{\dfrac{\partial }{\partial p_{i}^{a}}}}
\end{equation*}%
is an arbitrary d-vector field on the dual $1$-jet space $E^{\ast }=J^{1\ast
}(\mathcal{T},M)$.
\end{itemize}
\end{theorem}

\begin{proof}
Let $(Y_{A})$ and $(\omega ^{A})$, where $A\in \left\{ a,i,{_{\left(
i\right) }^{\left( a\right) }}\right\} $, be on $E^{\ast }=J^{1\ast }(%
\mathcal{T},M)$ the dual bases adapted to the nonlinear connection $N$, and
let $X=X^{F}Y_{F}$ be a d-vector field on $E^{\ast }$. In this context,
using the following true equalities (applied for a $CD\Gamma (N)$-linear
connection $D$):\medskip

\begin{enumerate}
\item $D_{Y_{C}}Y_{B}=\Gamma _{BC}^{F}Y_{F},\medskip $

\item $[Y_{B},Y_{C}]=R_{BC}^{F}Y_{F},\medskip $

\item $\mathbb{T}(Y_{C},Y_{B})=\mathbb{T}_{BC}^{F}Y_{F}=\{\Gamma
_{BC}^{F}-\Gamma _{CB}^{F}-R_{CB}^{F}\}Y_{F},\medskip $

\item $\mathbb{R}(Y_{C},Y_{B})Y_{A}=\mathbb{R}_{ABC}^{F}Y_{F},\medskip $

\item $D_{Y_{C}}\omega ^{B}=-\Gamma _{FC}^{B}\omega ^{F},\medskip $

\item $[\mathbb{R}(Y_{C},Y_{B})X]\otimes \omega ^{B}\otimes \omega
^{C}=\{D_{Y_{C}}D_{Y_{B}}X-\medskip $

$-D_{Y_{B}}D_{Y_{C}}X-D_{[Y_{C},Y_{B}]}X\}\otimes \omega ^{B}\otimes \omega
^{C},\medskip$
\end{enumerate}

by a direct calculation, we find that%
\begin{equation}
X_{:B:C}^{A}-X_{:C:B}^{A}=X^{F}\mathbb{R}_{FBC}^{A}-X_{:F}^{A}\mathbb{T}%
_{BC}^{F},  \label{ric}
\end{equation}%
where \textquotedblright $_{:G}$\textquotedblright\ represents one from the
local covariant derivatives \textquotedblright $_{/b}$\textquotedblright ,
\textquotedblright $_{|j}$\textquotedblright\ or \textquotedblright $%
|_{(b)}^{(j)}$\textquotedblright\ produced by the $h$-normal $N$-linear
connection of Cartan type $CD\Gamma (N)$.

Taking into account in (\ref{ric}) that the indices $A,B,C,\ldots $ belong
to the set%
\begin{equation*}
\left\{ a,i,{_{\left( i\right) }^{\left( a\right) }}\right\} ,
\end{equation*}%
and using the particular features of an $h$-normal $N$-linear connection of
Cartan type $CD\Gamma (N)$ (i.e., the torsion d-components $T_{jk}^{i}$ are
zero; we have the curvature relations from the Table (\ref{Table2})), by
complicated computations, we find what we were looking for (see also the
Table (\ref{Table1})).
\end{proof}

In order to find an interesting application of the preceding Ricci
identities, let us consider the \textit{canonical Liouville-Hamilton
d-tensor field of polymomenta }on $E^{\ast }=J^{1\ast }(T,M)$, which is
given by%
\begin{equation*}
\mathbb{C}^{\ast }{=p_{i}^{a}{\dfrac{\partial }{\partial p_{i}^{a}}}}.
\end{equation*}%
In this context, for an $h$-normal $N$-linear connection of Cartan type $%
CD\Gamma (N)$, we can construct the \textit{non-metrical deflection d-tensors%
}, setting 
\begin{equation*}
\Delta _{(i)b}^{(a)}={p_{i/b}^{a}},\quad \Delta _{(i)j}^{(a)}={p_{i|j}^{a}}%
,\quad \vartheta _{(i)(b)}^{(a)(j)}={p_{i}^{a}}|_{(b)}^{(j)},
\end{equation*}%
where "$_{/b}$", "$_{|j}$" and "$|_{(b)}^{(j)}$" are the local covariant
derivatives produced by $CD\Gamma (N)$.

By direct local computations, we deduce that the non-metrical deflection
d-tensors of $CD\Gamma (N)$ have the expressions:%
\begin{equation*}
\begin{array}{c}
\medskip \Delta _{(i)b}^{(a)}=-\underset{1}{N}\overset{\left( a\right) }{%
_{\left( i\right) b}}-A_{ib}^{r}p_{r}^{a}+\chi _{fb}^{a}p_{i}^{f},\quad
\Delta _{(i)j}^{(a)}=-\underset{2}{N}\overset{\left( a\right) }{_{\left(
i\right) j}}-H_{ij}^{r}p_{r}^{a}, \\ 
\vartheta _{(i)(b)}^{(a)(j)}=\delta _{b}^{a}\delta
_{i}^{j}-C_{i(b)}^{r(j)}p_{r}^{a}.%
\end{array}%
\end{equation*}

Applying now the preceding $(v)$-set of Ricci identities (attached to an $h$%
-normal $N$-linear connection of Cartan type) to the components of the
canonical Liouville-Hamilton d-vector field\textit{\ }of polymomenta, we get

\begin{corollary}
The following the \textbf{deflection d-tensor identities}, associated to an $%
h$-normal N-linear connection of Cartan type, are true: 
\begin{equation}
\left\{ 
\begin{array}{l}
\medskip \Delta _{(i)b/c}^{(a)}-\Delta
_{(i)c/b}^{(a)}=p_{r}^{a}R_{ibc}^{r}-p_{i}^{f}\chi _{fbc}^{a}-\vartheta
_{(i)(f)}^{(a)(r)}R_{(r)bc}^{(f)} \\ 
\medskip \Delta _{(i)b|k}^{(a)}-\Delta
_{(i)k/b}^{(a)}=p_{r}^{a}R_{ibk}^{r}-\Delta
_{(i)r}^{(a)}T_{bk}^{r}-\vartheta _{(i)(f)}^{(a)(r)}R_{(r)bk}^{(f)} \\ 
\medskip \Delta _{(i)j|k}^{(a)}-\Delta
_{(i)k|j}^{(a)}=p_{r}^{a}R_{ijk}^{r}-\vartheta
_{(i)(f)}^{(a)(r)}R_{(r)jk}^{(f)} \\ 
\medskip \Delta _{(i)b}^{(a)}|_{(c)}^{(k)}-\vartheta
_{(i)(c)/b}^{(a)(k)}=p_{r}^{a}P_{ib(c)}^{r\;\;(k)}-\vartheta
_{(i)(f)}^{(a)(r)}P_{(r)b(c)}^{(f)\;\;(k)} \\ 
\medskip \Delta _{(i)j}^{(a)}|_{(c)}^{(k)}-\vartheta
_{(i)(c)|j}^{(a)(k)}=p_{r}^{a}P_{ij(c)}^{r\;\ (k)}-\Delta
_{(i)r}^{(a)}C_{j(c)}^{r(k)}-\vartheta
_{(i)(f)}^{(a)(r)}P_{(r)j(c)}^{(f)\;(k)} \\ 
\medskip \vartheta _{(i)(b)}^{(a)(j)}|_{(c)}^{(k)}-\vartheta
_{(i)(c)}^{(a)(k)}|_{(b)}^{(j)}=p_{r}^{a}S_{i(b)(c)}^{r(j)(k)}-\vartheta
_{(i)(f)}^{(a)(r)}S_{(r)(b)(c)}^{(f)(j)(k)}.%
\end{array}%
\right.  \label{defl_ID_h-normal_Cartan}
\end{equation}
\end{corollary}

\begin{remark}
The deflection d-tensor identities (\ref{defl_ID_h-normal_Cartan}) will be
used in the near future for the construction of the \textbf{geometrical
Maxwell equations} that will govern the abstract multi-time geometrical
"electromagnetism" produced by a quadratic Hamiltonian depending on
polymomenta (this is our work in progress).
\end{remark}

\section{The local Bianchi identities of the $CD\Gamma (N)$-connections on
the dual jet bundle $J^{1\ast }(T,M)$}

From the general theory of linear connections on a vector bundle, one knows
that the torsions $\mathbb{T}$ and curvature $\mathbb{R}$ of a connection $D$
on the dual $1$-jet space $E^{\ast }=J^{1\ast }(T,M)$ are not independent.
In other words, they are interrelated by the following general \textit{%
Bianchi identities} (for any $X,Y,Z,U\in \mathcal{X}\left( E^{\ast }\right) $%
):%
\begin{equation*}
\begin{array}{c}
\sum\limits_{\left\{ X,Y,Z\right\} }\left\{ \left( D_{X}\mathbb{T}\right)
\left( Y,Z\right) -\mathbb{R}(X,Y)Z+\mathbb{T}\left( \mathbb{T}%
(X,Y),Z\right) \right\} =0,\medskip \\ 
\sum\limits_{\left\{ X,Y,Z\right\} }\left( D_{X}\mathbb{R}\right) \left(
Y,Z,U\right) +\mathbb{R}\left( \mathbb{T}(X,Y),Z\right) U=0,%
\end{array}%
\end{equation*}%
where $\Sigma _{\left\{ X,Y,Z\right\} }$ means a cyclic sum. Obviously,
working with a $CD\Gamma (N)$-linear connection and the local adapted basis
of d-vector fields $\left( X_{A}\right) \subset \mathcal{X}\left( E^{\ast
}\right) $ (associated to the given nonlinear connection $N$ on $E^{\ast }$%
), the above Bianchi identities are locally described by the equalities:%
\begin{equation}
\begin{array}{c}
\sum\limits_{\left\{ A,B,C\right\} }\left\{ \mathbb{R}_{ABC}^{F}-\mathbb{T}%
_{AB:C}^{F}-\mathbb{T}_{AB}^{G}\mathbb{T}_{CG}^{F}\right\} =0,\medskip \\ 
\sum\limits_{\left\{ A,B,C\right\} }\left\{ \mathbb{R}_{DAB:C}^{F}+\mathbb{T}%
_{AB}^{G}\mathbb{R}_{DCG}^{F}\right\} =0,%
\end{array}
\label{Bianchi_local}
\end{equation}%
where $\mathbb{R}(X_{A},X_{B})X_{C}=\mathbb{R}_{CBA}^{D}X_{D}$, $\mathbb{T}%
(X_{A},X_{B})=\mathbb{T}_{BA}^{D}X_{D},$ and \textquotedblright $_{:C}$%
\textquotedblright\ represents one from the local covariant derivatives
\textquotedblright $_{/a}$\textquotedblright , \textquotedblright $_{|i}$%
\textquotedblright\ or \textquotedblright $|_{(a)}^{(i)}$\textquotedblright\
of the $CD\Gamma (N)$-linear connection $D$ (for similar details, see the
works \cite{Miro+Anas}, \cite{Miro+Hrim+Shim+Saba} and \cite{Ne9}).
Consequently, we find:

\begin{theorem}
The following \textbf{thirty} effective local \textbf{Bianchi identities}
for an $h$-normal $N$-linear connection of Cartan type $CD\Gamma (N)$ are
true on the dual $1$-jet space $E^{\ast }=J^{1\ast }(\mathcal{T},M)$:$%
\medskip $

\begin{itemize}
\item the first set:$\medskip $

\begin{enumerate}
\item[\emph{1.}] $\sum_{\{a,b,c\}}\chi _{abc}^{d}=0,\medskip $

\item[\emph{2.}] $\mathcal{A}_{\{a,b\}}\left\{
T_{ar}^{l}T_{bk}^{r}-T_{ak/b}^{l}\right\}
=R_{kab}^{l}-C_{k(f)}^{l(r)}R_{(r)ab}^{(f)},\medskip$

\item[\emph{3.}] $\mathcal{A}_{\{j,k\}}\left\{
C_{k(f)}^{l(r)}R_{(r)aj}^{(f)}+R_{jak}^{l}+T_{aj|k}^{l}\right\} =0,\medskip$

\item[\emph{4.}] $\sum_{\{i,j,k\}}\left\{
C_{k(f)}^{l(r)}R_{(r)ij}^{(f)}-R_{ijk}^{l}\right\} =0,\medskip$
\end{enumerate}

\item the second set:\medskip

\begin{enumerate}
\item[\emph{5.}] $\sum_{\{a,b,c\}}\left\{
R_{(l)ab/c}^{(d)}+P_{(l)c(f)}^{(d)\;\;(r)}R_{(r)ab}^{(f)}\right\}
=0,\medskip $

\item[\emph{6.}] $\mathcal{A}_{\{a,b\}}\left\{
R_{(l)ak/b}^{(d)}+P_{(l)b(f)}^{(d)\;%
\;(r)}R_{(r)ak}^{(f)}+R_{(l)br}^{(d)}T_{ak}^{r}\right\} =\medskip$

$=R_{(l)ab|k}^{(d)}+P_{(l)k(f)}^{(d)\;\;(r)}R_{(r)ab}^{(f)},\medskip$

\item[\emph{7.}] $\mathcal{A}_{\{j,k\}}\left\{
R_{(l)aj|k}^{(d)}+P_{(l)k(f)}^{(d)\;\
(r)}R_{(r)aj}^{(f)}+R_{(l)kr}^{(d)}T_{aj}^{r}\right\} =\medskip$

$=-R_{(l)jk/a}^{(d)}-P_{(l)a(f)}^{(d)\;\;(r)}R_{(r)jk}^{(f)},\medskip$

\item[\emph{8.}] $\sum_{\{i,j,k\}}\left\{
R_{(l)ij|k}^{(d)}+P_{(l)k(f)}^{(d)\;\ (r)}R_{(r)ij}^{(f)}\right\}
=0,\medskip $
\end{enumerate}

\item the third set:\medskip

\begin{enumerate}
\item[\emph{9.}] $%
T_{ak}^{l}|_{(e)}^{(p)}-C_{r(e)}^{l(p)}T_{ak}^{r}+P_{ka(e)}^{l\;%
\;(p)}+C_{k(e)/a}^{l(p)}-C_{k(f)}^{l(r)}P_{(r)a(e)}^{(f)\;%
\;(p)}+C_{k(e)}^{r(p)}T_{ar}^{l}=0,\medskip$

\item[\emph{10.}] $\mathcal{A}_{\{j,k\}}\left\{
C_{j(e)|k}^{l(p)}+C_{k(f)}^{l(r)}P_{(r)j(e)}^{(f)\;\ (p)}+P_{jk(e)}^{l\;\
(p)}\right\} =0,\medskip$
\end{enumerate}

\item the fourth set:\medskip

\begin{enumerate}
\item[\emph{11.}] $\mathcal{A}_{\{a,b\}}\left\{
P_{(l)a(e)/b}^{(d)\;\;(p)}+P_{(l)b(f)}^{(d)\;\;(r)}P_{(r)a(e)}^{(f)\;\;(p)}%
\right\} =\medskip$

$%
=R_{(l)ab}^{(d)}|_{(e)}^{(p)}+R_{(l)(e)ab}^{(d)(p)}+S_{(l)(e)(f)}^{(d)(p)(r)}R_{(r)ab}^{(f)},\medskip 
$

\item[\emph{12.}] $\mathcal{A}_{\{a,k\}}\left\{
P_{(l)a(e)|k}^{(d)\;\;(p)}+P_{(l)k(f)}^{(d)\;\;(r)}P_{(r)a(e)}^{(f)\;\;(p)}%
\right\} =\medskip $

$%
=R_{(l)ak}^{(d)}|_{(e)}^{(p)}+R_{(l)(e)ak}^{(d)(p)}+S_{(l)(e)(f)}^{(d)(p)(r)}R_{(r)ak}^{(f)}+R_{(l)ar}^{(d)}C_{k(e)}^{r(p)}-T_{ak}^{r}P_{(l)r(e)}^{(d)\;(p)}\medskip 
$

\item[\emph{13.}] $\mathcal{A}_{\{j,k\}}\left\{
P_{(l)j(e)|k}^{(d)\;(p)}+P_{(l)k(f)}^{(d)\;(r)}P_{(r)j(e)}^{(f)%
\;(p)}+R_{(l)kr}^{(d)}C_{j(e)}^{r(p)}\right\} =\medskip$

$%
=R_{(l)jk}^{(d)}|_{(e)}^{(p)}+R_{(l)(e)jk}^{(d)(p)}+S_{(l)(e)(f)}^{(d)(p)(r)}R_{(r)jk}^{(f)},\medskip 
$
\end{enumerate}

\item the fifth set:\medskip

\begin{enumerate}
\item[\emph{14.}] $\mathcal{A}_{\left\{ {_{(b)}^{(j)}},{_{(c)}^{(k)}}%
\right\} }\left\{
C_{i(b)}^{l(j)}|_{(c)}^{(k)}+C_{i(c)}^{r(k)}C_{r(b)}^{l(j)}\right\}
=S_{i(b)(c)}^{l(j)(k)}-C_{i(f)}^{l(r)}S_{(r)(b)(c)}^{(f)(j)(k)},\medskip$
\end{enumerate}

\item the sixth set:\medskip

\begin{enumerate}
\item[\emph{15.}] $\mathcal{A}_{\left\{ {_{(b)}^{(j)}},{_{(c)}^{(k)}}%
\right\} }\left\{
P_{(l)a(b)}^{(d)\;\;(j)}|_{(c)}^{(k)}+P_{(r)a(b)}^{(f)\;%
\;(j)}S_{(l)(c)(f)}^{(d)(k)(r)}-P_{(l)(b)a(c)}^{(d)(j)\;\;(k)}\right\}
=\medskip$

$=-S_{(l)(b)(c)/a}^{(d)(j)(k)}-S_{(r)(b)(c)}^{(f)(j)(k)}P_{(l)a(f)}^{(d)\;%
\;(r)},\medskip$

\item[\emph{16.}] $\mathcal{A}_{\left\{ {_{(b)}^{(j)}},{_{(c)}^{(k)}}%
\right\} }\left\{
P_{(l)i(b)}^{(d)\;(j)}|_{(c)}^{(k)}+P_{(r)i(b)}^{(f)%
\;(j)}S_{(l)(c)(f)}^{(d)(k)(r)}-P_{(l)(b)i(c)}^{(d)(j)\;(k)}-\right.
\medskip $

$\left. -C_{i(b)}^{r(j)}P_{(l)r(c)}^{(d)\;\ (k)}\right\}
=-S_{(l)(b)(c)|i}^{(d)(j)(k)}-S_{(r)(b)(c)}^{(f)(j)(k)}P_{(l)i(f)}^{(d)%
\;(r)},\medskip$
\end{enumerate}

\item the seventh set:\medskip

\begin{enumerate}
\item[\emph{17.}] $\sum_{\left\{ {_{(a)}^{(i)}},{_{(b)}^{(j)}},{_{(c)}^{(k)}}%
\right\} }\left\{
S_{(l)(a)(b)}^{(d)(i)(j)}|_{(c)}^{(k)}+S_{(r)(a)(b)}^{(f)(i)(j)}S_{(l)(c)(f)}^{(d)(k)(r)}+S_{(l)(a)(b)(c)}^{(d)(i)(j)(k)}\right\} =0,\medskip 
$
\end{enumerate}

\item the eight set:\medskip

\begin{enumerate}
\item[\emph{18.}] $\sum_{\{a,b,c\}}\chi _{eab/c}^{d}=0,\medskip$

\item[\emph{19.}] $\chi _{eab|k}^{d}=0,\medskip$

\item[\emph{20.}] $\chi _{eab}^{d}|_{(c)}^{(k)}=0,\medskip$

\item[\emph{21.}] $\sum_{\{a,b,c\}}\left\{
R_{pab/c}^{l}+R_{(r)ab}^{(f)}P_{pc(f)}^{l\;\;(r)}\right\} =0,\medskip$

\item[\emph{22.}] $\mathcal{A}_{\{a,b\}}\left\{
R_{pak/b}^{l}+R_{(r)ak}^{(f)}P_{pb(f)}^{l\;\;(r)}+T_{ak}^{r}R_{pbr}^{l}%
\right\} =R_{pab|k}^{l}+R_{(r)ab}^{(f)}P_{pk(f)}^{l\;\ (r)},\medskip$

\item[\emph{23.}] $\mathcal{A}_{\{j,k\}}\left\{
R_{paj|k}^{l}+R_{(r)aj}^{(f)}P_{pk(f)}^{l\;\;(r)}+T_{aj}^{r}R_{pkr}^{l}%
\right\} =-R_{pjk/a}^{l}-R_{(r)jk}^{(f)}P_{pa(f)}^{l\;\ (r)},\medskip$

\item[\emph{24.}] $\sum_{\{i,j,k\}}\left\{
R_{pij|k}^{l}+R_{(r)ij}^{(f)}P_{pk(f)}^{l\;\ (r)}\right\} =0,\medskip$
\end{enumerate}

\item the nineth set:\medskip

\begin{enumerate}
\item[\emph{25.}] $\mathcal{A}_{\{a,b\}}\left\{
P_{ia(e)/b}^{l\;\;(p)}+P_{(r)a(e)}^{(f)\;\;(p)}P_{ib(f)}^{l\;(r)}\right\}
=R_{iab}^{l}|_{(e)}^{(p)}+R_{(r)ab}^{(f)}S_{i(e)(f)}^{l(p)(r)},\medskip$

\item[\emph{26.}] $\mathcal{A}_{\{a,k\}}\left\{
P_{ia(e)|k}^{l\;\;(p)}+P_{(r)a(e)}^{(f)\;\;(p)}P_{ik(f)}^{l\;\;(r)}\right\}
=\medskip$

$%
=R_{iak}^{l}|_{(e)}^{(p)}+R_{(r)ak}^{(f)}S_{i(e)(f)}^{l(p)(r)}+C_{k(e)}^{r(p)}R_{iar}^{l}-T_{ak}^{r}P_{ir(e)}^{l\;\;(p)},\medskip 
$

\item[\emph{27.}] $\mathcal{A}_{\{j,k\}}\left\{
P_{ij(e)|k}^{l\;\;(p)}+P_{(r)j(e)}^{(f)\;\;(p)}P_{ik(f)}^{l\;%
\;(r)}+C_{j(e)}^{r(p)}R_{ikr}^{l}\right\} =\medskip$

$=R_{ijk}^{l}|_{(e)}^{(p)}+R_{(r)jk}^{(f)}S_{i(e)(f)}^{l(p)(r)},\medskip$
\end{enumerate}

\item the tenth set:\medskip

\begin{enumerate}
\item[\emph{28.}] $\mathcal{A}_{\left\{ {_{(b)}^{(j)}},{_{(c)}^{(k)}}%
\right\} }\left\{
P_{pa(b)}^{l\;\;(j)}|_{(c)}^{(k)}+P_{(r)a(b)}^{(f)\;%
\;(j)}S_{p(c)(f)}^{l(k)(r)}\right\} =\medskip$

$=-S_{p(b)(c)/a}^{l(j)(k)}-S_{(r)(b)(c)}^{(f)(j)(k)}P_{pa(f)}^{l\;\;(r)},%
\medskip$

\item[\emph{29.}] $\mathcal{A}_{\left\{ {_{(b)}^{(j)}},{_{(c)}^{(k)}}%
\right\} }\left\{
P_{pi(b)}^{l\;\;(j)}|_{(c)}^{(k)}+P_{(r)i(b)}^{(f)\;%
\;(j)}S_{p(c)(f)}^{l(k)(r)}-C_{i(b)}^{r(j)}P_{pr(c)}^{l\;\;(k)}\right\}
=\medskip$

$=-S_{p(b)(c)|i}^{l(j)(k)}-S_{(r)(b)(c)}^{(f)(j)(k)}P_{pi(f)}^{l\;\;(r)},%
\medskip$
\end{enumerate}

\item the eleventh set:\medskip

\begin{enumerate}
\item[\emph{30.}] $\sum_{\left\{ {_{(a)}^{(i)}},{_{(b)}^{(j)}},{_{(c)}^{(k)}}%
\right\} }\left\{
S_{p(a)(b)}^{l(i)(j)}|_{(c)}^{(k)}+S_{(f)(a)(b)}^{(r)(i)(j)}S_{p(c)(r)}^{l(k)(f)}\right\} =0,\medskip 
$
\end{enumerate}

where, if $\{A,B,C\}$ are indices of type $\left\{ a,i,{_{(i)}^{(a)}}%
\right\} $, then $\sum_{\{A,B,C\}}$ represents a cyclic sum, and $\mathcal{A}%
_{\{A,B\}}$ represents an alternate sum.
\end{itemize}
\end{theorem}

\begin{proof}
Taking into account that the indices $A,B,C,D...$ are of type%
\begin{equation*}
\left\{ a,i,{_{\left( i\right) }^{\left( a\right) }}\right\} ,
\end{equation*}%
and the torsion $\mathbb{T}_{AB}^{C}$ and curvature $\mathbb{R}_{ABC}^{D}$
adapted components are given in the Tables (\ref{Table1}) and (\ref{Table2}%
), after laborious local computations, the formulas (\ref{Bianchi_local})
imply the required Bianchi identities.
\end{proof}

\begin{remark}
We point out that, in the particular single-time case%
\begin{equation*}
(\mathcal{T},h)=(\mathbb{R},\delta =1),
\end{equation*}%
the last identity of our each set of local Bianchi identities reduces to one
of the classical eleven Bianchi identities that characterize the $N$\textit{%
-linear connections} in the classical Hamilton geometry on cotangent bundles
(see \emph{\cite{Miro+Hrim+Shim+Saba}}).
\end{remark}

\textbf{Acknowledgements.} The authors of this paper would like to express
their sincere gratitude to Professor Gh. Atanasiu for his suggestions and
useful discussions on this research topic.

Alexandru OAN\u{A} and Mircea NEAGU

University Transilvania of Bra\c{s}ov,

Department of Mathematics - Informatics,

Blvd. Iuliu Maniu, no. 50, Bra\c{s}ov 500091, Romania.

\textit{E-mails}: alexandru.oana@unitbv.ro, mircea.neagu@unitbv.ro

\end{document}